\documentclass{amsart}
\usepackage{latexsym}
\usepackage{amsmath}
\usepackage{amsfonts}
\usepackage{amsthm}
\usepackage{amssymb}
\usepackage[T1]{fontenc}
\usepackage{hyperref}
\usepackage{amsrefs}

\newtheorem{thm}{Theorem}
\newtheorem{cor}{Corollary}
\theoremstyle{remark}
\newtheorem*{rem}{Remark}
\newtheorem{ex}{Example}
\theoremstyle{definition}
\newtheorem*{df}{Definition}
\newcommand{\R}{\mathbb{R}}
\newcommand{\Q}{\mathbb{Q}}
\newcommand{\Z}{\mathbb{Z}}
\newcommand{\N}{\mathbb{N}}
\renewcommand{\S}{\mathcal{S}}

\newcommand{\eps}{\varepsilon}
\newcommand{\f}{\varphi}

\renewcommand{\(}{\left(} \renewcommand{\)}{\right)}
\renewcommand{\[}{\left[} \renewcommand{\]}{\right]}

\begin{document}

\title[Characterizations of derivations]{Characterizations of derivations on spaces of smooth functions}

\author{W\l{}odzimierz Fechner}
\address{Institute of Mathematics, Lodz University of Technology, al. Politechniki 8, 93-590 \L\'od\'z, Poland}
\email{wlodzimierz.fechner@p.lodz.pl}

\author{Aleksandra \'Swi\k{a}tczak}
\address{Institute of Mathematics, Lodz University of Technology, al. Politechniki 8, 93-590 \L\'od\'z, Poland}
\email{aleswi97@gmail.com}

\begin{abstract}
We provide a list of equivalent conditions under which an additive operator acting on a space of smooth functions on a compact real interval is a multiple of the derivation.
\end{abstract}



\maketitle 

\section{Introduction}

By $\R$ we denote the set of reals, $\Q$ are rationals, $\Z$ are integers, $\N=\{1, 2, \ldots \}$ and $\N_0=\N \cup \{0 \}$. If $I\subseteq \R$ is an interval and $k\in \N_0$, then $C^k(I)$ is the space o real-valued functions on $I$ that are  $k$-times continuously differentiable on the interior of $I$. If $k=0$, then we write simply $C(I)$. The space $C^k(I)$ is furnished with the standard pointwise algebraic operations and hence it is a real commutative algebra. 

\medskip

\begin{df}[e.g. M. Kuczma \cite{Kuczma}*{page 391}]
Assume that $Q$ is a commutative ring and $P$ is a subring of $Q$. Function $f\colon P \to Q$ is called \emph{derivation} if it is \emph{additive}:
\begin{equation}
f(x+y) = f(x) + f(y), \quad x, y \in P
\label{add}
\end{equation}
and it satisfies the \emph{Leibniz rule}:
\begin{equation}
f(xy) = xf(y) + yf(x), \quad x, y \in P.
\label{der}
\end{equation}
\end{df}
The following theorem describes derivations over fields of characteristic zero.

\begin{thm}[\cite{Kuczma}*{Theorem 14.2.1}]
Let $K$  be a field of characteristic zero, $F$ be a subfield of $K$, $S$ be an algebraic base of $K$ over $F$ if it exists, and let $S = \varnothing$ otherwise. If $f \colon F \to K$ is a derivation, then, for every function $u\colon S \to K$ there exists a unique derivation $g\colon K \to K$ such that $g=f$ on $F$ and $g=u$ on $S$. 
\end{thm}

From this theorem it follows in particular that nonzero derivations $f\colon\R\to\R$ exist. It is well known they are discontinuous and very irregular mappings. For an exhaustive discussion of the notion of the derivation and related functional equations the reader is referred to E. Gselmann \cites{G1, G2}, E. Gselmann, G. Kiss, C. Vincze \cite{G3} and references therein. Recently B. Ebanks \cites{E1, E2} studied derivations and derivations of higher order on rings.

\medskip

The ''model'' example of a derivation is the operator of derivative on the space $C^k(I)$ for $k>0$ . Indeed, if we define  $T\colon C^k(I)\to C(I)$ as $T(f) = f'$ for $f \in C^k(I)$, then clearly $ C^k(I)$ is a subring of $C(I)$,  $T$ is additive and it satisfies the Leibniz rule:
\begin{equation}
T(f\cdot g) = f\cdot T(g) + g \cdot T(f).
\label{LR}
\end{equation}

Crucial results about equation \eqref{LR} on the space $C^k(I)$  are due to H. K\"{o}nig and V. Milman. We refer the reader to their recent monograph \cite{KM}. They studied several operator equations and inequalities that are related to the derivatives on the spaces of smooth functions. Later on, we will utilize their elegant result   \cite{KM}*{Theorem 3.1}  regarding \eqref{LR}. Briefly, if $I$ is an open set, then the general solution of \eqref{LR} for all $ f, g \in C^k(I)$ is of the form
\begin{equation}
T(f) =  c \cdot f\cdot\ln|f| + d \cdot f',\quad f \in C^k(I)
\label{k}
\end{equation}
for some continuous functions $c, d\in C(I)$, if $k>0$, and
\begin{equation}
T(f) =  c \cdot f\cdot\ln|f| ,\quad f \in C^k(I)
\label{0}
\end{equation}
if $k=0$ (in formulas \eqref{k} and \eqref{0} the convention that $0\cdot \ln 0 = 0$ is adopted). Note that no additivity is assumed.

\medskip

There is a natural question to characterize real-to-real derivations among additive functions with the aid of a relation which is weaker than \eqref{der}. In particular, the very first article published in the first volume of \emph{Aequationes Mathematicae} by A. Nishiyama and S. Horinouchi \cite{NH} addresses this question.
The authors studied the following relations, each of them is a direct consequence of \eqref{der} alone and together with \eqref{add} implies \eqref{der}:
\begin{equation}
f(x^2) = 2xf(x), \quad x \in \R,
\label{2}
\end{equation}
\begin{equation}
f(x^{-1}) = -x^{-2}f(x), \quad x \in \R,\,  x \neq 0,
\label{-1}
\end{equation}
and
\begin{equation}
f(x^n) = ax^{n-m}f(x^m), \quad x \in \R, \,  x \neq 0,
\label{n}
\end{equation}
where $a\neq 1$ and $n, m$ are integers such that $am = n\neq 0$. Further similar results, as well as some generalizations, are due to W. Jurkat \cite{J},  Pl. Kannappan and S. Kurepa \cites{KK1, KK2}, S. Kurepa \cite{K}, among others. B. Ebanks \cite{E3} generalized and extended these results to arbitrary fields. A recent paper by M. Amou \cite{A} provides some $n$-dimensional generalizations of the results of \cites{J, KK1, KK2, K}.

\medskip

This paper provides versions of the above-mentioned results for operators $T\colon C^k(I)\to C(I)$. Therefore, we seek conditions which are equivalent to \eqref{LR}.

\section{Main results}

Throughout this section let us fix $k\in \N_0$ and an interval $I\subseteq \R$.
We will study conditions upon an additive operator $T\colon C^k(I)\to C(I)$ which yield analogues to equations \eqref{2}, \eqref{-1} and \eqref{n}. Therefore, we will focus on the following operator relations:
\begin{equation}
T(f^2) = 2f \cdot T(f),
\label{T2}
\end{equation}
\begin{equation}
T(f) = -f^{2}\cdot T\(\frac{1}{f}\), 
\label{T-1}
\end{equation}
\begin{equation}
T(f^{n}) = nf^{n-1}\cdot T(f).
\label{Tn}
\end{equation}

Our first theorem is a simple observation that some reasonings concerning derivations from the real-to-real case can be extended to arbitrary commutative rings without substantial changes. We adopted parts of proof of \cite{Kuczma}*{Theorem 14.3.1}.

\begin{thm}\label{t1}
Assume that $Q$ is a commutative ring, $P$ is a subring of $Q$ and $T\colon P\to Q$ is an additive operator. Then, the following conditions are pairwise equivalent:
\begin{itemize}
	\item[$(i)$] $T$ satisfies $T(f^2) = 2f \cdot T(f)$ for all $f \in P$,
	\item[$(ii)$] $T$ satisfies $T(f\cdot g) = f\cdot T(g) + g \cdot T(f)$ for all $f, g \in P$,
	\item[$(iii)$] $T$ satisfies $T(f^{n}) = nf^{n-1}\cdot T(f)$ for all $f \in P$ and $n \in \N$.
\end{itemize}
\end{thm}
\begin{proof}
$(i)\Rightarrow (ii)$. 
Fix arbitrarily $f, g \in P$. By \eqref{T2} we get
$$
T((f+g)^2)= 2(f+g)\cdot T(f+g).
$$
Since $T$ is additive, then
$$ T(f^2) + 2T(f\cdot g) + T(g^2) = 2f\cdot T(f) + 2g \cdot T(f) + 2f\cdot T(g) + 2g\cdot T(g).$$
Using \eqref{T2} again, after reductions we obtain \eqref{LR}.

$(ii)\Rightarrow (iii)$. If $n=1$, then \eqref{Tn} reduces to an identity. Assume that \eqref{Tn} holds for some $n \in \N$ and all $f \in P$. Then, by \eqref{LR} and the induction hypothesis we have
\begin{align*}
T(f^{n+1}) &= T(f^n \cdot f) = f^n \cdot T(f) + f \cdot T(f^n) \\&=f^n\cdot T(f) + n f^{n-1+1} \cdot T(f) = (n+1)f^n\cdot T(f).
\end{align*}
$(iii)\Rightarrow (i)$. Take $n=2$.
\end{proof}

The next corollary will be utilized later on.

\begin{cor}\label{c1}
Assume that $T\colon C^k(I)\to C(I)$ is an additive operator. Then, the following conditions are pairwise equivalent:
\begin{itemize}
	\item[$(i)$] $T$ satisfies $T(f^2) = 2f \cdot T(f)$ for all $f \in C^k(I)$,
	\item[$(ii)$] $T$ satisfies $T(f\cdot g) = f\cdot T(g) + g \cdot T(f)$ for all $f, g \in C^k(I)$,
	\item[$(iii)$] $T$ satisfies  $T(f^{n}) = nf^{n-1}\cdot T(f)$ for all $f \in C^k(I)$ and $n \in \N$.
\end{itemize}
\end{cor}

Our next result characterizes the Leibniz rule \eqref{LR} on a domain restricted to functions separated from zero. Thus, we can consider conditions \eqref{T-1} and \eqref{Tn} for negative $n$, which involve the function $1/f$. The situation is a bit more complicated, but Theorem \ref{t2} below has a mainly technical role.

\begin{thm}\label{t2}
Assume that $T\colon C^k(I)\to C(I)$ is an additive operator and $\eps_1\in(0,1)$, $\eps_2\in(0,1)$ and $c\in (1, +\infty]$ are constants. Consider the following conditions:
\begin{itemize}
  \item[$(i)$]  $T$ satisfies $T(f) = -f^{2}\cdot T\(\frac{1}{f}\)$ for all $f \in C^k(I)$, $c>f>\eps_1$,
	\item[$(ii)$]  $T$ satisfies $T(f^2) = 2f \cdot T(f)$ for all $f \in C^k(I)$, $f>\eps_2$,
	\item[$(iii)$]  $T$ satisfies $T(f\cdot g) = f\cdot T(g) + g \cdot T(f)$ for all $f, g \in C^k(I)$, $f>\eps_2$, $g>\eps_2$,
	\item[$(iv)$]  $T$ satisfies  $T(f^{n}) = nf^{n-1}\cdot T(f)$ for all  $n \in \Z$ and all $f \in C^k(I)$ such that $\eps_2<f<1/\eps_2$,  and $f^{n-1}>\eps_2$ if $n>0$ and 
	$f^{n+1}>\eps_2$ if $n<0$. 
\end{itemize}
Then: $(i)$ with $c=+\infty$ implies $(ii)$ with $\eps_2> \sqrt{\eps_1}$,  $(ii)$ and $(iii)$ are equivalent, $(iii)$ implies $(iv)$, $(iv)$ implies $(i)$ with $\eps_1=\eps_2$ and $c= 1/\eps_2$.
\end{thm}
\begin{proof}
$(i)\Rightarrow (ii)$. First, note that by applying \eqref{T-1} for $f=1$ and using the rational homogeneity of $T$ we get that $T$ vanishes on each constant function equal to a rational number. Observe that for arbitrary rational $\delta >0$ (which will be chosen later) the identity
\begin{equation}
\frac{1}{f^2-\delta^2} = \frac{1}{2\delta}\( \frac{1}{f-\delta}- \frac{1}{f+\delta} \)
\label{delta}
\end{equation}
holds for $f \in C^k(I)$ such that $f>\delta$. Next, if $\eps_1>0$ is given and $\eps_2> \sqrt{\eps_1}$, then we will find some rational $\delta>0$ such that $\eps_2>\eps_1+\delta$ and $\eps_2^2>\eps_1+\delta^2$. Consequently, if $f \in C^k(I)$ and $f>\eps_2$, then $f\pm \delta >\eps_1$ and $f^2-\delta ^2 >\eps_1$. Using $(i)$ three times together with \eqref{delta} and the additivity of $T$ we obtain
\begin{align*}
T(f^2)&=T(f^2-\delta^2)= -(f^2-\delta^2)^2T\(\frac{1}{f^2-\delta^2}\) \\&= -\frac{1}{2\delta}(f^2-\delta^2)^2T\( \frac{1}{f-\delta}- \frac{1}{f+\delta}\) \\&= -\frac{1}{2\delta}(f+\delta)^2(f-\delta)^2\[T\( \frac{1}{f-\delta}\)- T\(\frac{1}{f+\delta}\)\] \\&=\frac{1}{2\delta}\[(f+\delta)^2T(f-\delta) - (f-\delta)^2T(f+\delta)\] = 2fT(f).
\end{align*}
$(ii)\Leftrightarrow (iii)$. Analogously as in Theorem \ref{t1} for $f>\eps_2$ and $g>\eps_2$.
\\
$(iii)\Rightarrow (iv)$. If $n=1$, then \eqref{Tn} is trivially satisfied. Assume that $f$, $n$ and $\eps_2$ satisfy assumptions of $(iv)$. For $n>1$ we proceed like in Theorem \ref{t1}. If $n=0$, then $(iv)$ reduces to $T(1)=0$, which follows from $(iii)$. If $n=-1$, then for $1/\eps_2>f>\eps_2$ we have
$$0=T(1) = T\(f\cdot \frac{1}{f}\)= \frac{1}{f}\cdot T(f) + f\cdot T\(\frac{1}{f}\).$$  

Assume that $n<-1$. By downward induction, one can check that for $f^{n+1}>\eps_2$ we have from \eqref{LR}
\begin{align*} 
T(f^n) &= T\(f^{n+1}\cdot \frac{1}{f}\) = f^{n+1}\cdot T\(\frac{1}{f}\) + \frac{1}{f}\cdot T\(f^{n+1}\) \\&=-f^{n+1}\cdot f^{-2} T(f) + \frac{n+1}{f}\cdot f^n\cdot T\(f\) = nf^{n-1}T(f).
\end{align*}
$(iv)\Rightarrow (i)$. Take $n=-1$.
\end{proof}

If we assume additionally that interval $I$ is compact, then the situation clarifies considerably. 

\begin{thm}\label{t3}
Assume that $I$ is compact and $T\colon C^k(I)\to C(I)$ is an additive operator. Then, the following conditions are pairwise equivalent:
\begin{itemize}
  \item[$(i)$] $T$ satisfies  $T(f\cdot g) = f\cdot T(g) + g \cdot T(f)$ for all $f, g \in C^k(I)$,
	\item[$(ii)$] $T$ satisfies  $T(f\cdot g) = f\cdot T(g) + g \cdot T(f)$ for all $f, g \in C^k(I), f>0$, $g>0$,
	\item[$(iii)$] $T$ satisfies $T(f^2) = 2f \cdot T(f)$  for all $f \in C^k(I)$, 
	\item[$(iv)$] $T$ satisfies  $T(f^2) = 2f \cdot T(f)$  for all $f \in C^k(I)$, $f>0$, 
	\item[$(v)$]  $T$ satisfies $T(f) = -f^{2}\cdot T\(\frac{1}{f}\)$ for all $f \in C^k(I)$, $f>0$,
	\item[$(vi)$] $T$ satisfies $T(f^{n}) = nf^{n-1}\cdot T(f)$ for all $f \in C^k(I)$ and $n \in \N$,
	\item[$(vii)$] $T$ satisfies $T(f^{n}) = nf^{n-1}\cdot T(f)$ for all $f \in C^k(I)$, $f>0$ and $n \in \N$.
\end{itemize}
\end{thm}
\begin{proof}
This statement is a consequence of  Corollary \ref{c1} and Theorem \ref{t2}. Since $I$ is compact, then  $f$ attains its global extrema. Thus, we will find some rational $r, q \in \Q$ such that $1/2<rf+q<2$.
Moreover, as it was already observed in the proof of Theorem \ref{t2}, each of the conditions of Theorem \ref{t3} implies that $T(1)=0$ and then $T$ vanishes on constant function equal to a rational number. 
Consequently, we have $T(rf+q)=rT(f) + T(q) = rT(f)$ and therefore Theorem \ref{t2} applies for the conditions $(ii)$, $(iv)$, $(v)$ and $(vii)$ with appropriately chosen $\eps_1$ and $\eps_2$. The remaining conditions are equivalent by Corollary \ref{c1}. Therefore, we are done if we prove for example the implication $(iv)\Rightarrow (iii)$. 

Fix $f\in C^k(I)$ arbitrarily and choose $r, q \in \Q$ such that $1/2<rf+q<2$. By $(iv)$ we get
$$T((rf+q)^2)= 2(rf+q) T(rf+q).$$
Then using additivity we obtain
$$r^2T(f^2) + 2rqT(f) + T(q^2) = 2r^2fT(f) + 2rq T(f)$$
and after reduction
$$T(f^2) + 0 = 2fT(f)$$
i.e. condition $(iii)$.
\end{proof}

One can join Corollary \ref{c1} and Theorem \ref{t3} with the mentioned result of H. K\"{o}nig and V. Milman to obtain a corollary. 

\begin{cor}\label{c2}
Under assumptions of Corollary \ref{c1} or Theorem \ref{t3}, if $k>0$, then each of the conditions listed there is equivalent to the following one:
\begin{itemize}
	\item[$(x)$] there exists some $d\in C(I)$ such that $T(f) = d\cdot f'$ for all $f \in C^k(I)$
\end{itemize}
and if $k=0$, then $T=0$ is the only additive operator that fulfils any of the equivalent conditions.
\end{cor}
\begin{proof}
Consider $f(x) = x$ on $I$ and denote $\tilde{d}:= T(f)\in C(I)$. Next, note that by \cite{KM}*{Theorem 3.1} the formulas \eqref{k} and \eqref{0}, respectively hold on the interior of $I$ with some $c, d \in C(\mathrm{int} I)$. The additivity of $T$ implies that $c=0$. Therefore $\tilde{d}$ is a continuous extension of $d$ to the whole interval $I$.
\end{proof}

\section{Final remarks}

\begin{rem}
Inequalities between $f$, $g$ and constants $\eps_1$ and $\eps_2$ in Theorem \ref{t2} are not optimal. This however was not our goal since the role of this result is auxiliary only. Similarly, inequality $f>0$ in some of the conditions of Theorem \ref{t3} can be equivalently replaced by an estimate from above or from below by any other fixed constant.

Moreover, in the proof of Theorem \ref{t3} we showed more than is stated. Namely, it is equivalently enough to assume instead $f>0$ that $f$ is bilaterally bounded by two rational numbers, like $1/2$ and $2$. However, since this generalization is apparent only and easy, we do not include it in the formulation of the theorem.
\end{rem}

\begin{ex}
Assume that $\f\colon (1, \infty) \to \R$ is a smooth mapping that satisfies equation
\begin{equation}
\f(2x) = 2 \f(x), \quad x \in (1, \infty).
\label{dwa}
\end{equation} Such mappings exist in abundance. In fact, every map $\f_0$ defined on $(1,2]$ can be uniquely extended to a solution of \eqref{dwa}. Next, let $d\colon (e, \infty) \to \R$ be defined as
$$
d(x) = x \cdot \f(\ln x), \quad x \in (e, \infty).
$$
We see easily that 
$$d(x^2) = 2x d(x), \quad x \in (e, \infty)
$$
and 
$$d(xy) \neq x d(y) + yd(x)$$ in general, unless $\f$ is additive.
Define $T\colon C^1((e, \infty))\to C((e, \infty))$ as follows:
$$ T(f) = d \circ f , \quad f \in C^1((e, \infty)).$$
One can see that $T$ satisfies \eqref{T2} for all $f, g \in C((e, \infty))$, but fails to satisfy the Leibniz rule \eqref{LR}. 
Thus, the assumption of additivity in our all results is essential.
Observe also that $T$ has the property that it vanishes on constant functions equal to a rational. This fact, as a consequence of additivity, was frequently used in the proofs of our Theorems \ref{t2} and \ref{t3}. Therefore, the additivity assumption cannot be relaxed to this property.
\end{ex}

\begin{ex}
Assume that $I$ is an interval and $T$ is given by the formula
$$ T(f) = f'' - \frac{(f')^2}{f}, \quad f \in C^2(I), \, f>0. $$
Then $T$ satisfies \eqref{LR} for all $f, g \in C^2(I)$ such that $f>0$ and $g>0$. This observation is a particular case of the second part of \cite{KM}*{Corollary 3.4}. Clearly, $T$ is not additive. Moreover, $T$ cannot be extended in such a way it satisfies \eqref{LR} on the whole space $C^2(I)$.
\end{ex}

The following examples show that if the domain of operator $T$ is changed, then the conditions discussed in our results are no longer equivalent and various situations are possible.

\begin{ex}
Let $\S$ be the space of all functions $f \in C^1((0, \infty))$ which satisfy functional equation
\begin{equation}
f(x+1) = 2f(x), \quad x \in (0, \infty).
\label{+1}
\end{equation}
Note that $\S$ is not closed under multiplication. Moreover, each function $f_0\colon (0, 1] \to \R$ can be uniquely extended to a solution of \eqref{+1}. Therefore, $\S$ is an infinite-dimensional subspace of $C^1((0, \infty))$. Define $T\colon C^1((0, \infty)) \to C^1((0, \infty))$ by the formula
$$T(f)(x) = f(x+1), \quad f \in C^1((0, \infty)), \, x \in (0, \infty).$$
It is easy to check that $T$ is additive and satisfies \eqref{LR} for $f, g \in \S$. Thus, there are more solutions of \eqref{LR} if the domain of $T$ is restricted to a particular subspace of $C^k(I)$.
\end{ex}

\begin{ex}
Let $P[x]$ be the space of all real polynomials of variable $x$. By $\deg(f)$ we denote the degree of a polynomial $f\in P[x]$. Define $T\colon P[x] \to P[x]$ by
$$T(f) = \deg (f) \cdot f, \quad f \in P[x].$$
Then $T$ is not additive, it satisfies \eqref{LR} and has no extension to a solution of \eqref{LR} to $C^k(\R)$.
\end{ex}

\begin{ex}
Let $$\S:=\{f\colon (0, \infty)\to\R : f(x) = x^k \textrm{ for some } k \in \Z \textrm{ and } x \in (0, \infty)   \}.$$
Note that $\S$ is closed under multiplication but it is not a linear space. Next, let a double sequence $\f$ on $\Z$ of natural numbers be defined as follows: $\f(0)=0$,  $\f(k)$ is arbitrary but $\neq k$ if $k$ is odd, and if $k=2^n \cdot m$ with some $n \in \N$ and odd $m\in \Z$, then $$\f(k) := 2^{\frac{n^2-n}{2}}\cdot m^n \cdot \f (m).$$ Note that we have
\begin{align}
\f(2k) &= \f(2^{n+1}\cdot m) = 2^{\frac{n^2+n}{2}} \cdot m^{n+1} \cdot  \f(m) \nonumber \\&= 2^n \cdot m \cdot 2^{\frac{n^2-n}{2}} \cdot m^n \cdot  \f(m) = k\cdot \f(k), \quad k \in \Z.
\label{2k}
\end{align}
Define $T\colon\S\to C((0, \infty))$ by
\begin{equation}
T(f)(x)  := k \cdot x^{\f(k)} , \quad x \in (0, \infty)
\label{dupa}
\end{equation}
if $f(x) = x^k$ for $ x \in (0, \infty)$. One can see that if $f$ is of this form, then by \eqref{2k}
$$T(f^2)(x) = 2k \cdot x^{\f(2k)} = 2k \cdot x^{k\cdot \f(k)} = 2f (x) T(f)(x)$$ for all $x \in (0, \infty)$, i.e. $T$ satisfies \eqref{T2}. 

Moreover, one can see that \eqref{T-1} is equivalent to the equality
$$\f(k) - \f(-k) = 2k, \quad k \in \Z, \, k \neq 0.$$
Therefore, we can construct a sequence $\f$ such that $T$ defined by \eqref{dupa} satisfies \eqref{T-1} as well as another sequence $\f'$ for which $T$ does not satisfy \eqref{T-1}. 
Finally, \eqref{LR} is not true on $\S$. Indeed, note that if \eqref{LR} is satisfied by $T$ given by \eqref{dupa}, then:
$$ \f(k+l) = \f(k) + l = \f(l) + k, \quad k, l \in \Z, k \neq 0, l \neq 0,$$
which does not hold.
\end{ex}


\begin{bibdiv}
\begin{biblist}


		

		

\bib{A}{article}{
   author={Amou, Masaaki},
   title={Multiadditive functions satisfying certain functional equations},
   journal={Aequationes Math.},
   volume={93},
   date={2019},
   number={2},
   pages={345--350},
   issn={0001-9054},
}

\bib{E1}{article}{
   author={Ebanks, Bruce},
   title={Derivations and Leibniz differences on rings},
   journal={Aequationes Math.},
   volume={93},
   date={2019},
   number={3},
   pages={629--640},
   issn={0001-9054},
}

\bib{E2}{article}{
   author={Ebanks, Bruce},
   title={Derivations and Leibniz differences on rings: II},
   journal={Aequationes Math.},
   volume={93},
   date={2019},
   number={6},
   pages={1127--1138},
   issn={0001-9054},
}
		

\bib{E3}{article}{
   author={Ebanks, Bruce},
   title={Functional equations characterizing derivations and homomorphisms on fields},
   journal={Results Math.},
   volume={74},
   date={2019},
   number={4},
   pages={Paper No. 146, 12},
   issn={1422-6383},
}
		


		

\bib{G1}{article}{
   author={Gselmann, Eszter},
   title={Notes on the characterization of derivations},
   journal={Acta Sci. Math. (Szeged)},
   volume={78},
   date={2012},
   number={1-2},
   pages={137--145},
   issn={0001-6969},
}

\bib{G2}{article}{
   author={Gselmann, Eszter},
   title={Characterizations of derivations},
   journal={Dissertationes Math.},
   volume={539},
   date={2019},
   pages={65},
   issn={0012-3862},
}
		

\bib{G3}{article}{
   author={Gselmann, Eszter},
   author={Kiss, Gergely},
   author={Vincze, Csaba},
   title={On functional equations characterizing derivations: methods and
   examples},
   journal={Results Math.},
   volume={73},
   date={2018},
   number={2},
   pages={Paper No. 74, 27},
   issn={1422-6383},
}
		


		

\bib{J}{article}{    author={Jurkat, Wolfgang B.},    title={On Cauchy's functional equation},    journal={Proc. Amer. Math. Soc.},    volume={16},    date={1965},    pages={683--686},    issn={0002-9939},   
}


\bib{KK1}{article}{
   author={Kannappan, Palaniappan},
   author={Kurepa, Svetozar},
   title={Some relations between additive functions. I},
   journal={Aequationes Math.},
   volume={4},
   date={1970},
   pages={163--175},
   issn={0001-9054},
}

\bib{KK2}{article}{
   author={Kannappan, Palaniappan},
   author={Kurepa, Svetozar},
   title={Some relations between additive functions. II},
   journal={Aequationes Math.},
   volume={6},
   date={1971},
   pages={46--58},
   issn={0001-9054},
}

\bib{KM}{book}{
   author={K\"{o}nig, Hermann},
   author={Milman, Vitali},
   title={Operator relations characterizing derivatives},
   publisher={Birkh\"{a}user/Springer, Cham},
   date={2018},
   pages={vi+191},
   isbn={978-3-030-00240-4},
   isbn={978-3-030-00241-1},
}

\bib{Kuczma}{book}{
   author={Kuczma, Marek},
   title={An introduction to the theory of functional equations and
   inequalities},
   edition={2},
   note={Cauchy's equation and Jensen's inequality;
   Edited and with a preface by Attila Gil\'{a}nyi},
   publisher={Birkh\"{a}user Verlag, Basel},
   date={2009},
   pages={xiv+595},
   isbn={978-3-7643-8748-8},
}

\bib{K}{article}{
   author={Kurepa, Svetozar},
   title={The Cauchy functional equation and scalar product in vector
   spaces},
   language={English, with Serbo-Croatian summary},
   journal={Glasnik Mat.-Fiz. Astronom. Dru\v{s}tvo Mat. Fiz. Hrvatske Ser. II},
   volume={19},
   date={1964},
   pages={23--36},
   issn={0367-4789},
}

\bib{NH}{article}{
   author={Nishiyama, Akinori},
   author={Horinouchi, S\^{o}ichi},
   title={On a system of functional equations},
   journal={Aequationes Math.},
   volume={1},
   date={1968},
   pages={1--5},
   issn={0001-9054},
}



\end{biblist}
\end{bibdiv}
\end{document}